\documentclass[11pt]{amsart}

\usepackage{amssymb,amsmath,color,hyperref}
\usepackage[mathscr]{eucal}

\theoremstyle{plain}
\newtheorem{thm}{Theorem}[section]
\newtheorem{theorem}[thm]{Theorem}

\newtheorem{prop}[thm]{Proposition}

\theoremstyle{definition}

\theoremstyle{remark}

\newtheorem{remark}{Remark}
\newtheorem*{remark*}{Remark}

\numberwithin{equation}{section}

        \newcommand{\field}[1]{{\mathbb{#1}}}
        \newcommand{\NN}{\field{N}}
        \newcommand{\ZZ}{\field{Z}}
        
        \newcommand{\RR}{\field{R}}

\begin{document}

\title[Determinantal point processes]{Determinantal point processes associated with the Bochner-Schr\"odinger operator}

\author[Y. A. Kordyukov]{Yuri A. Kordyukov}
\address{Institute of Mathematics, Ufa Federal Research Centre, Russian Academy of Sciences, 112~Chernyshevsky str., 450008 Ufa, Russia} \email{yurikor@matem.anrb.ru}

%
%
\begin{abstract}
We consider the Bochner-Schr\"odinger operator $$H_{p}=\frac 1p\Delta^{L^p}+V$$ on tensor powers $L^p$ of a Hermitian line bundle $L$ on a Riemannian manifold $X$ of bounded geometry under the assumption of non-degeneracy of the curvature form of $L$. For large $p$, the spectrum of $H_p$ asymptotically coincides with the union $\Sigma$ of all local Landau levels of the operator at the points of $X$. We study the determinantal point process on $X$ associated with the spectral projection of $H_p$ corresponding to an interval $I=(\alpha,\beta)$ such that $\alpha,\beta\not \in \Sigma$ and compute the asymptotics of its linear statistics as $p$ goes to infinity. When $X$ is compact, this implies the law of large numbers and central limit theorem for the corresponding empirical measures. 
\end{abstract}


 \maketitle
\section{Introduction}
\subsection{General setting}
Let $(X,g)$ be a complete Riemannian manifold of dimension $d$ and $(L,h^L)$ a Hermitian line bundle on $X$ with a Hermitian connection $\nabla^L$. We suppose that $(X, g)$ is  a manifold of bounded geometry and $L$ has bounded geometry. This means that the curvatures $R^{TX}$ and $R^L$ of the Levi-Civita connection $\nabla^{TX}$ and connection $\nabla^L$, respectively, and their derivatives of any order are uniformly bounded on $X$ in the norm induced by $g$ and $h^L$, and the injectivity radius $r_X$ of $(X, g)$ is positive.

For any $p\in \NN$, let $L^p:=L^{\otimes p}$ be the $p$th tensor power of $L$ and let
\[
\nabla^{L^p}: {C}^\infty(X,L^p)\to
{C}^\infty(X, T^*X \otimes L^p)
\] 
be the Hermitian connection on $L^p$ induced by $\nabla^{L}$. Consider the induced Bochner Laplacian $\Delta^{L^p}$ acting on $C^\infty(X,L^p)$ by
\[
\Delta^{L^p}=\big(\nabla^{L^p}\big)^{\!*}\,
\nabla^{L^p},
\] 
where $\big(\nabla^{L^p}\big)^{\!*}: {C}^\infty(X,T^*X\otimes L^p)\to
{C}^\infty(X,L^p)$ is the formal adjoint of  $\nabla^{L^p}$. Let $V\in C^\infty(X,\RR)$ be a real-valued function. 

We study the Bochner-Schr\"odinger operator $H_p$ acting on $C^\infty(X,L^p)$ by
\[
H_{p}=\frac 1p\Delta^{L^p}+V. 
\] 
The operator $H_p$ is essentially self-adjoint  in the Hilbert space $L^2(X,L^p)$ with initial domain  $C^\infty_c(X,L^p)$. We still denote by $H_p$ its unique self-adjoint extension and by $\sigma(H_p)$ its spectrum in $L^2(X,L^p)$.

Consider the real-valued closed 2-form $\mathbf B$ (the magnetic field) given by 
\[
\mathbf B=iR^L. 
\] 
where $R^L$ is the curvature of the connection $\nabla^L $ defined as $R^L=(\nabla^L)^2$. 
We assume that $\mathbf B$ is non-degenerate. Thus, $X$ is a symplectic manifold. In particular, its dimension is even, $d=2n$, $n\in \NN$.

\begin{remark}
 Assume that the Hermitian line bundle $(L,h^L)$ is trivial. Then we can write $\nabla^L=d-i \mathbf A$ with a real-valued 1-form $\mathbf A$ (the magnetic potential), and we have
\[
R^L=-id\mathbf A,\quad \mathbf B=d\mathbf A. 
\]
The operator $H_p$ is related with the semiclassical magnetic Schr\"odinger operator
\[
H_p=\hbar^{-1}[(i\hbar d+\mathbf A)^*(i\hbar d+\mathbf A)+\hbar V], \quad \hbar=\frac{1}{p},\quad p\in \NN. 
\]
It can be also considered as the magnetic Schr\"odinger operator with strong electric and magnetic fields, growing at the same rate: 
\[
H_p=\frac{1}{p}[(d-ip\mathbf A)^*(d-ip\mathbf A)+pV], \quad p\in \NN. 
\]
\end{remark} 

\begin{remark} 
If $X$ is the Euclidean space $\RR^{2n}$ with coordinates $Z=(Z_1,\ldots,Z_{2n})$, we can write the 1-form $\bf A$ as
\[
{\bf A}= \sum_{j=1}^{2n}A_j(Z)\,dZ_j,
\]
the matrix of the Riemannian metric as $g(Z)=(g_{j\ell}(Z))_{1\leq j,\ell\leq 2n}$
and its inverse as $g(Z)^{-1}=(g^{j\ell}(Z))_{1\leq j,\ell\leq 2n}$.
Denote $|g(Z)|=\det(g(Z))$. Then $\bf B$ is given by 
\[
{\bf B}=\sum_{j<k}B_{jk}\,dZ_j\wedge dZ_k, \quad
B_{jk}=\frac{\partial A_k}{\partial Z_j}-\frac{\partial
A_j}{\partial Z_k}.
\]
Moreover, the operator $H_p$ has the form
\[
H_p=\frac{1}{p}\frac{1}{\sqrt{|g|}}\sum_{1\leq j,\ell\leq 2n}\left(i \frac{\partial}{\partial Z_j}+pA_j\right) \left[\sqrt{|g|} g^{j\ell} \left(i \frac{\partial}{\partial Z_\ell}+pA_\ell\right)\right]+V.
\]
Our assumptions hold, if the matrix $(B_{j\ell}(Z))$ has full rank $2n$ and its eigenvalues are separated from zero uniformly on $Z\in \RR^{2n}$, for any $\alpha \in \ZZ^{2n}_+$ and $1\leq j,\ell\leq 2n$, we have 
\[
\sup_{Z\in \RR^{2n}}|\partial^\alpha g_{j\ell}(Z)|<\infty, \quad \sup_{Z\in \RR^{2n}}|\partial^\alpha B_{j\ell}(Z)|<\infty, 
\]
and the matrix $(g_{j\ell}(Z))$ is positive definite uniformly on $Z\in \RR^{2n}$.
 \end{remark}

As shown in \cite{Kor22} (see also \cite{charles20b}), the spectrum of $H_{p}$ as $p\to \infty$ admits the following asymptotic description in terms of the spectra of the model operators.

For an arbitrary $x_0\in X$, the model operator at $x_0$ is a second order differential operator $\mathcal H^{(x_0)}_{p}$, acting on $C^\infty(T_{x_0}X)$, which is obtained from the operator $H_p$ by freezing coefficients at $x_0$.  

For $Y\in C^\infty(T_{x_0}X, T(T_{x_0}X))$ and $p\in\mathbb N$, introduce the covariant derivative
\[
\nabla^{(x_0)}_{p,Y} : C^\infty(T_{x_0}X)\to C^\infty(T_{x_0}X)
\] 
by the formula
\begin{equation}\label{eq1.1}
\nabla^{(x_0)}_{p,Y}=\nabla_Y-ip\alpha^{(x_0)}(Y), 
\end{equation}
where $\nabla_Y$ denotes the differentiation operator on $T_{x_0}X$ in the direction $Y$ and the connection one-form $\alpha^{(x_0)}\in C^\infty(T_{x_0}X,T^*(T_{x_0}X))$ is given by 
\begin{equation}\label{eq1.2}
\alpha^{(x_0)}_v(w)=\frac{1}{2}\mathbf B_{x_0}(v,w),\quad v\in T_{x_0}X,\quad w\in T_v(T_{x_0}X), 
\end{equation}
and we identify $v\in T_{x_0}X$ with the corresponding element in $T_v(T_{x_0}X)$. 

The curvature of $\nabla^{(x_0)}_p$ is constant: $d\alpha^{(x_0)}=\mathbf B_{x_0}$.  Denote by $\Delta^{(x_0)}_{p}$ the associated Bochner Laplacian. The model operator $\mathcal H^{(x_0)}_{p}$ acting on $C^\infty(T_{x_0}X)$ is defined as 
\[
\mathcal H^{(x_0)}_{p}=\frac 1p\Delta^{(x_0)}_p+V(x_0).
\]

For ${x_0}\in X$, let $B_{x_0} : T_{x_0}X\to T_{x_0}X$ be the skew-symmetric operator such that 
\[
\mathbf B_{x_0}(u,v)=g(B_{x_0}u,v), \quad u,v\in T_{x_0}X. 
\]
Its eigenvalues have the form $\pm i a_j(x_0), j=1,\ldots,n,$ with $a_j(x_0)>0$.  It is well-known that the spectrum of $\mathcal H^{(x_0)}_{p}$ is independent of $p$ and consists of eigenvalues of infinite multiplicity (Landau levels): 
\[
\sigma(\mathcal H^{(x_0)}_{p})=\Sigma_{x_0}:=\left\{\Lambda_{\mathbf k}({x_0})\,:\, \mathbf k\in\ZZ_+^n \right\}, 
\]
where, for $\mathbf k=(k_1,\cdots,k_n)\in\ZZ_+^n$ and $x_0\in X$,
\begin{equation}\label{eq1.3}
\Lambda_{\mathbf k}(x_0)=\sum_{j=1}^n(2k_j+1) a_j(x_0)+V(x_0).
\end{equation}
In particular, the lowest eigenvalue of $\mathcal H^{(x_0)}_{p}$ is 
\[
\Lambda_0(x_0):=\sum_{j=1}^n a_j(x_0)+ V(x_0). 
\]
Let $\Sigma$ be the union of the spectra of the model operators: 
\[
\Sigma=\bigcup_{x\in X}\Sigma_x=\left\{\Lambda_\mathbf {k}(x)\,:\, \mathbf k\in\ZZ_+^n, x\in X \right\}.
\]

\begin{theorem}[\cite{Kor22}, Theorem 1]\label{th1.1}
For any $K>0$, there exists $c>0$ such that for any $p\in \NN$ the spectrum of $H_{p}$ in the interval  $[0,K]$  is  contained in the $cp^{-1/4}$-neighborhood of $\Sigma$.  
\end{theorem}

Take an interval $I=(\alpha,\beta)$ such that $\alpha,\beta\not \in \Sigma$. 
By Theorem \ref{th1.1}, there exists $\mu_0>0$ and $p_0\in \NN$ such that for any $p>p_0$ 
\[
\sigma(H_{p})\subset (-\infty, \alpha-\mu_0) \cup I \cup (\beta+\mu_0, \infty).
\] 
Let $P_{p,I}$ be the spectral projection of $H_{p}$ associated with $I$:
\[
P_{p,I} : L^2(X,L^p)\to \mathcal H_p=\operatorname{Im} P_{p,I}. 
\]
We will consider the manifold $X$ equipped with the Riemannian volume form $dv_X$ and the determinantal point process on $X$ associated with the projection $P_{p,I}$ for $p>p_0$. Since $P_{p,I}$ is an integral operator with smooth kernel, it is of locally trace class, and, therefore, the determinantal point process on $X$ associated with $P_{p,I}$ exists (see Theorem~\ref{th2.1} below). 

\subsection{Statement of the main results}
For any measurable function $f: X\to \mathbb R$, denote by $\mathcal N_p[f]$, $p>p_0$, the linear statistic of the determinantal point process on $X$ associated with the projection $P_{p,I}$ (see the formula \eqref{eq2.2} below). First, we compute the asymptotics of the distribution of $\mathcal N_p[f]$ as $p$ goes to infinity.

Let 
$$
\mathcal K_I:=\{\mathbf k\in \ZZ^n_+ : \Lambda_{\mathbf k}(x)\in I\}.
$$
Under current assumption on $I$, this set is independent of $x\in X$.

Fix $x\in X$. Recall that $\pm i a_j(x), j=1,\ldots,n,$ with $a_j(x)>0$ are the eigenvalues of $B_{x}$. Choose an orthonormal base $\{e_j : j=1,\ldots,2n\}$ in $T_{x}X$ such that  
\[
B_{x}e_{2k-1}=a_k(x)e_{2k}, \quad B_{x}e_{2k}=-a_k(x)e_{2k-1},\quad k=1,\ldots,n. 
\]
For any $m=1,\ldots,n$, introduce a function $I_m : \mathbb Z^n_+\times \mathbb Z^n_+\to \mathbb Z$, setting, for $\mathbf k,\mathbf k^\prime\in \mathbb Z^n_+$,
\begin{equation}
\begin{aligned} \label{eq1.4}
I_m(\mathbf k,\mathbf k)& =2k_m+1,\\
I_m(\mathbf k+\epsilon_j,\mathbf k)=I_m(\mathbf k,\mathbf k+\epsilon_j)& =-(k_j+1)\delta_{jm}, \quad j=1,\ldots,n, \\
I_m(\mathbf k,\mathbf k^\prime)& =0, \quad |\mathbf k-\mathbf k^\prime|>1,
\end{aligned}
\end{equation}
where $(\epsilon_1,\ldots,\epsilon_n)$ is the standard basis in $\mathbb Z^n$. 

Set                                             
\[
\alpha_m=\sum_{\mathbf k^\prime,\mathbf k^{\prime\prime}\in \mathcal K_I} I_m(\mathbf k^\prime,\mathbf k^{\prime\prime}).
\]                               

For any $f \in C^1(X)$, define 
\[
|df(x)|_{I}^2= \sum_{m=1}^{n} \frac{\alpha_m}{a_m(x)}   \left[\left(\nabla_{e_{2m-1}}f(x)\right)^2+\left(\nabla_{e_{2m}}f(x)\right)^2\right],
\]
where, for any $v\in T_xX$, $\nabla_vf(x)$ denotes the derivative of $f$ in the direction of $v$.  
One can show that the function $x\in X\mapsto |df(x)|_{I}^2$ is a well-defined continuous function on $X$,  which is smooth if $f \in C^\infty(X)$ (see the formula \eqref{eq4.3} below).

 Let $\Omega_{\mathbf B}=\frac{1}{n!} \mathbf B^n$ be the Liouville volume form on $X$. 

\begin{thm}\label{th1.2}
For any real-valued $f \in L^1(X,dv_X)$, the expectation of the linear statistics $\mathcal N_p[f]$ is finite for all $p>p_0$, and satisfies the following asymptotics as $p\to \infty$,
\begin{equation}\label{eq1.5}
\mathbb E[\mathcal N_p[f]] = \frac{p^n}{(2\pi)^n}|\mathcal K_I| \int_X f(x)\, \Omega_{\mathbf B}(x) + o(p^n) . 
\end{equation}
For any $f \in C^1_c(X, \mathbb R)$,  the variance of the linear statistics
$\mathcal N_p[f]$ satisfies the following asymptotics as $p\to \infty$,
\begin{equation}\label{eq1.6}
\mathbb V[\mathcal N_p[f]]=\frac{1}{4\pi}\frac{p^{n-1}}{(2\pi)^{n-1}}\int_X|df(x)|_{I}^2\,\Omega_{\mathbf B}(x)+ o(p^{n-1}).    
\end{equation}
\end{thm}

Now we assume that $X$ is compact. Then the projection $P_{p,I}$ is a finite rank projection. In this case, the associated determinantal point process is often called an orthogonal ensemble. It can be constructed as follows (see, for instance, \cite[Lemma 4.5.1]{HKPV09}). 

Set $N_p:=\dim \mathcal H_p$.  Let $\{s_j \in \mathcal H_p: j=1,\ldots, N_p\}$ be an orthonormal basis in $\mathcal H_p$. We define the Slater determinant as the section $\Psi_p \in C^\infty(X^{N_p}, (L^p)^{\boxtimes N_p})$ of $(L^p)^{\boxtimes N_p}$ over the $N_p$-fold product $X^{N_p}$ given for any $(x_1, x_2, \ldots, x_{N_p}) \in X^{N_p}$ by
\begin{equation}\label{eq1.7}
\Psi_p(x_1, x_2, \ldots, x_{N_p}) := \det(s_j(x_i))_{i,j=1}^{N_p}, 
\end{equation}
which does not depend on the choice of orthonormal basis of $\mathcal H_p$.

One can show that the measure $d\nu_{N_p}$ on $X^{N_p}$ given by
\[
d\nu_{N_p}:=\frac{1}{N_p!} |\Psi_p|^2_{h^{L^p}} dv_X^{N_p}, 
\]
defines a probability measure on $X^{N_p}$. 
 
For any $p \in \mathbb N$ and any $(x_1, x_2, \ldots, x_{N_p}) \in X^{N_p}$, the Slater determinant
\eqref{eq1.7} satisfies the formula
\[
|\Psi_p(x_1, x_2, \ldots, x_{N_p})|^2_{h^{L^p}} = \det(P_{p,I}(x_i, x_j))_{i,j=1}^{N_p}.
\]
Using this formula, one can check that $(X^{N_p}, d\nu_{N_p})$ is the determinantal process associated with $P_{p,I}$. 

For any function $f : X\to \mathbb R$ and any $p \in \mathbb N$, the associated linear statistics is the random variable over $(X^{N_p}, d\nu_{N_p})$ defined by 
\[
\mathcal N_p[f](x_1, x_2, \ldots, x_{N_p})=\sum_{j=1}^{N_p}f(x_j),\quad (x_1, x_2, \ldots, x_{N_p})\in X^{N_p}.
\]

\begin{thm}[Law of Large Numbers]\label{th1.3}
For any real-valued $f \in L^\infty(X)$, we have the following convergence in probability as $p\to \infty$,
\begin{equation}\label{eq1.8}
\frac{1}{N_p}\mathcal N_p[f]  \to\frac{1}{{\rm Vol}_{\mathbf B}(X)} \int_Xf(x) \Omega_{\mathbf B}(x),
\end{equation}
where
\[
{\rm Vol}_{\mathbf B}(X)=\int_X  \Omega_{\mathbf B}(x).
\]
\end{thm}

\begin{thm}[Central Limit Theorem]\label{th1.4}
For any $f \in C^1_c(X, \mathbb R)$, the random variable 
\[
N^{\frac{1}{2n}+\frac 12}_p\frac{\mathcal N_p[f]-\mathbb E[\mathcal N_p[f]]}{N_p}
\]
converges in distribution as $p\to \infty$ to a centered normal random variable $N(0,\sigma^2)$ with variance 
\begin{equation}\label{eq1.9}
\sigma^2=\frac{1}{4\pi}\frac{1}{\left(|\mathcal K_I|{\rm Vol}_{\mathbf B}(X)\right)^{(n-1)/n}} \int_X|df(x)|_{I}^2\,\Omega_{\mathbf B}(x).
\end{equation}
\end{thm}

Theorems \ref{th1.3} and \ref{th1.4} thus state that as $p\to\infty$, the determinantal point processes $(X^{N_p}, d\nu_{N_p})$ associated with the spectral projection $P_{p,I}$  admit the equilibrium measure $$\mu_{\rm eq}=\frac{1}{{\rm Vol}_{\mathbf B}(X)}\Omega_{\mathbf B}$$ on $X$ and satisfy a central limit theorem with fluctuations given by \eqref{eq1.9}.

\subsection{Discussion and related works}
The simplest example of the setting considered in the paper is given by the Euclidean plane $\mathbb R^{2}$ equipped with the standard Euclidean metric and with the trivial Hermitian line bundle $L$ with the Hermitian connection of constant curvature (cf. \eqref{eq1.1} and \eqref{eq1.2}). The Bochner-Schr\"odinger operator is the magnetic Schr\"odinger operator with constant magnetic field (the Landau Hamiltonian). As mentioned above, its spectrum consists of only eigenvalues with infinite multiplicity called Landau levels (see \eqref{eq1.3}). The eigenspace, corresponding to the lowest Landau level $\Lambda_0$, can be identified with the Bargmann-Fock space. The determinantal process associated with the lowest Landau level is the infinite Ginibre point process introduced in \cite{Ginibre65}. The infinite Ginibre point process is the limiting point process of the finite Ginibre ensembles formed by the eigenvalues of an $N\times N$ random matrix with independent and identically distributed complex Gaussians entries. The Ginibre ensemble also has a physical interpretation as electrons of two-dimensional one-component plasma, also known as the Coulomb gas or Jellium, at the special temperature. 
The asymptotics of the linear statistics of the finite Ginibre ensemble when the number of points $N$ goes to infinity were first obtained in \cite{RV07}. 

The determinantal point processes associated with higher Landau levels were introduced in \cite{shirai15}, where they are called Ginibre-type point processes. The corresponding eigenspaces are Fock spaces of polyanalytic functions. Similar determinantal point processes called polyanalytic Ginibre ensembles are discussed in \cite{HH13}. They belong to a more general class of point processes called Weyl-Heisenberg ensembles introduced in \cite{APRT17}. The asymptotics of the linear statistics for the polyanalytic Ginibre process are studied in \cite{HW19}. We refer to \cite{Abreu23,APRT17,FL22,MMM23} for more information on polyanalytic Ginibre point processes.

The case when $(X,g)$ is a compact K\"ahler manifold and $(L,h^L)$ is a positive Hermitian holomorphic line bundle on $X$ with the Chern connection $\nabla^L$ was first considered by Berman \cite{Berman14,Berman18}. In this case, the Bochner Laplacian coincides with twice the holomorphic Kodaira Laplacian, and the eigenspace corresponding to its lowest eigenvalue is the space of holomorphic sections of $L$. The associated correlation kernel is the Bergman kernel. In fact, in \cite{Berman14,Berman18}, the curvature of $(L,h^L)$ is not assumed to be smooth nor positive, and one simply restricts the convergence to the ``bulk'', a subset of $X$ where the curvature is positive. 

In this setting the limit of a large number $N$ of particles corresponds to the limit when the line bundle $L$ is replaced by a large tensor power $L^p$. When $X$ is the complex projective space, this is just a geometric formulation of the theory of (weighted) multivariate orthogonal polynomials, with the tensor power $p$ corresponding to the degree of the polynomials. 

The asymptotics of the linear statistics of the associated determinantal point process are computed in \cite{Berman18}. The law of large numbers and central limit theorem for the corresponding empirical measures are also established (see also \cite{Lemoine22} for closely related results and refinements). In \cite[Cor. 1.7]{Charles-Estienne20}, Charles and Estienne computed the asymptotics of the linear statistics with respect to the characteristic function of an open subset $U \subset X$ with smooth boundary in the setting of \cite{Berman14,Berman18}. In particular, they establish a central limit theorem in this case. In \cite{Ioos}, Ioos introduced finite dimensional determinantal point processes associated with a circle action on a prequantized K\"ahler manifold of bounded geometry, which can be considered as a generalization of the finite Ginibre ensembles, and computed the asymptotics of its linear statistics. Finally, we mention a recent study of the determinantal processes associated with general Berezin-Toeplitz operators on $\mathbb C$ \cite{DL25a}. 

The results of this paper extend of the results of Berman \cite{Berman18} on the asymptotics of the linear statistics for positive line bundles to the case of an arbitrary magnetic Schr\"odinger operator. The proofs are based on the asymptotic estimates of the Schwartz kernels of the spectral projection $P_{p,I}$ obtained in \cite{Kor22}. 

We consider the simplest case when the endpoints of the interval $I$ belong to spectral gaps of the operator. In this case,  the support of the equilibrium measure appearing in the law of large numbers \eqref{eq1.8}, which is often called the droplet, coincides with the whole space $X$. So the asymptotic results of this paper can be considered as asymptotics in the bulk. When $I$ is an arbitrary interval, the droplet is non-trivial and the contribution of the boundary of the droplet should be also considered. This will be done elsewhere.  

The paper is organized as follows. In Section~\ref{s2}, we recall some basic information on determinantal point processes.  In Section~\ref{s3}, we recall the results on the asymptotic behavior of the Schwartz kernels of the spectral projection $P_{p,I}$ obtained in \cite{Kor22}. In Section~\ref{s4}, we give the proofs of the main results of the paper. In Section~\ref{s5}, we consider some particular cases of the asymptotic formula \eqref{eq1.6} for the variance. 

\section{Determinantal point processes}\label{s2}

In this section, we recall some basic information on determinantal point processes.
Determinantal point processes were introduced by Macchi \cite{Mac75,Mac77} in the 70s as a mathematical model for fermions in quantum mechanics and have been widely studied in many
settings, see \cite{AGZ10,HKPV09} for references and background.

Let $E$ be a second countable locally compact Hausdorff space, $B$ the $\sigma$-algebra of Borel subsets and $\mu$ a Radon measure on $(E,B)$. We denote by $\Lambda$ the set of non-negative integer-valued Radon measures on $E$ equipped with the topology generated by the functions
\[
\xi\in E\mapsto \xi(U)\in \mathbb Z_+.
\]
Equivalently, one can think of $\Lambda$ as the space of locally finite configurations $\xi = (x_i)_{i=-\infty}^{\infty}$ of points in $E$ (roughly speaking, $\xi=\sum_i\delta_{x_i}$), satisfying, for any compact $K \subset E$, the condition 
\[
\#_K(\xi) := \#\{i\in \mathbb Z : x_i \in K\}< +\infty.
\]
A $\sigma$-algebra $\mathcal F$ of measurable subsets of $\Lambda$ is generated by the cylinder sets $C_nB =\{\xi\in \Lambda : \#_B(\xi) = n\}$, where $B$ is a Borel set with a compact closure and $n \in\mathbb Z_+$. Let $P$ be a probability measure on $(\Lambda,\mathcal F)$. A triple $(\Lambda,\mathcal F ,P)$ is called a random point process on $E$. 

A locally integrable function $\rho_k : E^k\to \mathbb R_+$ is called the $k$-point correlation function of a random point process $(\Lambda,\mathcal F ,P)$ if, for any Borel subsets $A_1,\ldots, A_m$ of $E$ and for any $k_i \in \mathbb Z_+, i =1,\ldots, m$, such that $\sum_{i=1}^mk_i=k$, the
following formula holds:
\[
\mathbb E\left(\prod_{i=1}^m \frac{(\#_{A_i})!}{(\#_{A_i}-k_i)!}\right) = \int_{A_1^{k_1}\times \ldots A^{k_m}_m}
\rho_k(x_1,\ldots, x_k) d\mu(x_1)\ldots d\mu(x_k).
\] 
In particular, for any disjoint Borel subsets $A_1,\ldots, A_k$ of $E$, we have
\[
\mathbb E\left(\prod_{i=1}^k \#_{A_i}\right)= \int_{A_1\times \ldots A_k}\rho_k(x_1,\ldots, x_k) d\mu(x_1)\ldots d\mu(x_k).
\]

Let $K:L^2(E, d\mu) \to L^2(E, d\mu)$ be a self-adjoint, locally trace-class operator defined by the integral kernel $K(x,y)$. Recall that $K$ is of locally trace class, if, for any compact set $A\subset E$, $K_A := \chi_AK\chi_A$ is of trace class, where $\chi_A$ is the multiplication operator of the indicator function of the set $A$. 

A random point process on $E$ is said to be determinantal (or fermion) if its $n$-point correlation functions are of the form
\begin{equation}\label{eq2.1}
\rho_k(x_1,\ldots, x_k)=\det(K(x_i,x_j)_{1\leq i,j\leq k}), \quad (x_1,\ldots, x_k)\in E^k. 
\end{equation}
The kernel $K(x,y)$ is called a correlation kernel.

The following theorem provides a sufficient condition on $K$ for the existence of the associated determinantal point process.

\begin{thm}[\cite{SoshnikovUMN}]\label{th2.1}
Let $K:L^2(E, d\mu) \to L^2(E, d\mu)$ be a self-adjoint, locally trace-class operator. Then $K$
determines a determinantal point process on $E$ if and only if $0 \leq K \leq 1$. If this random point process exists, then it is unique.
\end{thm}

For any real-valued measurable function $f : E\to \mathbb R$, the associated linear statistic, which is the random variable over $(\Lambda,\mathcal F ,P)$, is defined
by the function
\begin{equation}\label{eq2.2}
S_f(\xi)=\sum_{i=-\infty}^{\infty}f(x_i), \quad \xi= (x_i)_{i=-\infty}^{\infty}\in \Lambda.
\end{equation}
The distribution of the linear statistic $S_f$ is given by the following well-known fact (see, for instance, \cite{Soshnikov02}).

\begin{prop} \label{p2.2}
For any $f \in L^\infty(E, \mathbb R)$, the expectation and the variance of the random variable $S_f$ with respect to the determinantal point process associated with the operator $K$ are given by
\begin{equation}\label{eq2.3}
\mathbb E[S_f] = \int_E K(x, x) f(x) d\mu(x) 
\end{equation}
and
\[
\mathbb V[S_f] =\int_E K(x, x) f^2(x) d\mu(x)-\int_E \int_E |K(x, y)|^2 f(x)f(y)\, d\mu(x)\, d\mu(y).
\]
In particular, if $K$ is a projection, $K^2=K$, then 
\begin{equation}\label{eq2.4}
\mathbb V[S_f] = \frac 12 \int_E \int_E |K(x, y)|^2 (f(x)-f(y))^2\, d\mu(x)\, d\mu(y). 
\end{equation}
\end{prop}

In our case, the operator $K$ acts on sections of a Hermitian line bundle $(L,h^L)$ over a Riemannian  manifold $X$. An adaptation of the above results to this setting was made by Berman \cite[\S 5.1]{Berman18} (see also \cite{Ioos,Lemoine22}). First, we define 
\[
\det(K(x_i,x_j)_{1\leq i,j\leq k})=\sum_{\sigma\in \mathfrak S_k}\varepsilon(\sigma)\bigotimes_{i=1}^k K(x_i,x_{\sigma(i)})\in \mathbb C,
\]
using the canonical isomorphism
\[
L_{x_1}\otimes L^*_{x_{\sigma(1)}}\otimes \ldots \otimes L_{x_k}\otimes L^*_{x_{\sigma(k)}} \cong\mathbb C. 
\]
This justify the formula \eqref{eq2.1} and the definition of the determinantal process. 

The key fact in the proof of Theorem~\ref{th2.1} is the following Fredholm formula 
\[
\operatorname{Tr}(\wedge^k(K))=\frac{1}{k!}\int \det(K(x_i,x_j)_{1\leq i,j\leq k})dx_1\ldots dx_k, 
\]
whose validity in the current setting can be easily checked. 

Finally, for an adaptation of Proposition~\ref{p2.2}, see \cite[Lemma 6.2]{Berman18} or \cite[Proposition 4.4]{Ioos}

\section{Asymptotic behavior of the spectral projection}\label{s3}
In this section, we recall the results on the asymptotic behavior of the Schwartz kernels of the spectral projection $P_{p,I}$ obtained in \cite{Kor22}.

Recall that $P_{p,I}$ denotes the spectral projection of $H_{p}$ associated with the interval $I$. Let $P_{p,I}(x,x^\prime)$, $x,x^\prime\in X$, be its smooth Schwartz kernel with respect to the Riemannian volume form $dv_X$. In the case when $H_p=\frac{1}{p}\Delta_p$, where  $\Delta_p$ is the renormalized Bochner Laplacian introduced in \cite{Gu-Uribe} and $I=(\alpha,\beta)$ is a sufficiently small open neighborhood of $0$, the projection $P_{p,I}$ is called the generalized Bergman projection in \cite{ma-ma08}, since it generalizes the Bergman projection on complex manifolds. Its kernel is called the generalized Bergman kernel. 

First, we state the off-diagonal exponential estimate for $P_{p,I}(x,x^\prime)$, which is the analog of \cite[Theorem 1]{ma-ma15} and \cite[Theorem 1.2]{ko-ma-ma} for the (generalized) Bergman kernel. 

\begin{theorem}[\cite{Kor22},Theorem 2]\label{th3.1}
There exists $c>0$ such that for any $k\in \mathbb N$, there exists $C_k>0$ such that for any $p\in \mathbb N$, $x, x^\prime \in X$, we have
\[
\big|P_{p,I}(x, x^\prime)\big|_{{C}^k}\leq C_k p^{n+\frac{k}{2}}
e^{-c\sqrt{p} \,d(x, x^\prime)}.
\]
\end{theorem}

Here $d(x,x^\prime)$ is the geodesic distance and $|P_{p,I}(x, x^\prime)|_{{C}^k}$ denotes the pointwise ${C}^k$-seminorm of the section $P_{p,I}$ at a point $(x, x^\prime)\in X\times X$, which is the sum of the norms induced by $h^L$ and $g$ of the derivatives up to order $k$ of $P_{p,I}$ with respect to the connection $\nabla^{L^p}$ and the Levi-Civita connection $\nabla^{TX}$ evaluated at $(x, x^\prime)$.

Then we describe an asymptotic expansion of  $P_{p,I}$ as $p\to \infty$ in a fixed neighborhood of the diagonal (independent of $p$). Such kind of expansion is called full off-diagonal expansion following Ma-Marinescu's
book \cite[Chapter 4]{ma-ma:book}. 

First, we introduce normal coordinates near an arbitrary point $x_0\in X$. 
We denote by $B^{X}(x_0,r)$ and $B^{T_{x_0}X}(0,r)$ the open balls in $X$ and $T_{x_0}X$ with center $x_0$ and radius $r$, respectively. We identify $B^{T_{x_0}X}(0,r_X)$ with $B^{X}(x_0,r_X)$ via the exponential map $\exp^X_{x_0}: T_{x_0}X \to X$. Furthermore, we choose trivializations of the bundle $L$ over $B^{X}(x_0,r_X)$, identifying its fibers $L_Z$ at $Z\in B^{T_{x_0}X}(0,r_X)\cong B^{X}(x_0,r_X)$ with the space $L_{x_0}$ by parallel transport with respect to the connections $\nabla^L$ along the curve $\gamma_Z : [0,1]\ni u \to \exp^X_{x_0}(uZ)$. Denote by $\nabla^{L^p}$ and $h^{L^p}$ the connection and the Hermitian metric on the trivial bundle over $B^{T_{x_0}X}(0,r_X)$ with fiber $(L^p)_{x_0}$ induced by these trivializations.  

Let $dv_{TX}$ denote the Riemannian volume form of the Euclidean space $(T_{x_0}X, g_{x_0})$. We define a smooth function $\kappa_{x_0}$ on $B^{T_{x_0}X}(0,r_X)\cong B^{X}(x_0,r_X)$ by the equation
\[
dv_{X}(Z)=\kappa_{x_0}(Z)dv_{TX}(Z), \quad Z\in B^{T_{x_0}X}(0,r_X). 
\] 
The kernel $P_{p,I}(x,x^\prime)$ induces a smooth function $P_{p,I,x_0}(Z,Z^\prime)$ on the fiberwise product $$TX\times_X TX=\{(Z,Z^\prime)\in T_{x_0}X\times T_{x_0}X : x_0\in X\}$$ defined for all $x_0\in X$ and $Z,Z^\prime\in T_{x_0}X$ with $|Z|, |Z^\prime|<r_X$:
\[
P_{p,I,x_0}(Z,Z^\prime)=P_{p,I}(\exp^X_{x_0}(Z),\exp^X_{x_0}(Z^\prime)). 
\]

Denote by $\mathcal P_{I,x_0}$ the spectral projection of the model operator $\mathcal H^{(x_0)}:=\mathcal H^{(x_0)}_1$ associated with $I$. It can be written as
\begin{equation}\label{eq3.1}
\mathcal P_{I,x_0}=\sum_{\mathbf k \in \mathcal K_I} \mathcal P_{\Lambda_{\mathbf k},x_0},
\end{equation}
where $\mathcal P_{\Lambda_{\mathbf k},x_0}$ is the orthogonal projection to the eigenspace of the model operator $\mathcal H^{(x_0)}$ with the eigenvalue $\Lambda_{\mathbf k}$. 

One can give an explicit formula for the smooth Schwartz kernel of $\mathcal P_{\Lambda_{\mathbf k},x_0}$.
Choose an orthonormal base $\{e_j : j=1,\ldots,2n\}$ in $T_{x_0}X$ such that  
\[
B_{x_0}e_{2k-1}=a_k(x_0)e_{2k}, \quad B_{x_0}e_{2k}=-a_k(x_0)e_{2k-1},\quad k=1,\ldots,n. 
\]
It gives rise to the linear isomorphism $$T_{x_0}X\cong \mathbb R^{2n},\quad Z\in T_{x_0}X \mapsto (Z_1,\ldots,Z_{2n})\in \mathbb R^{2n}.$$ Set $z_j=Z_{2j-1}+iZ_{2j}, \bar z_j=Z_{2j-1}-iZ_{2j}, j=1,\ldots,n$. Then we have 
\begin{multline}\label{eq3.2}
\mathcal P_{\Lambda_{\mathbf k}}(Z,Z^\prime)=\frac{1}{(2\pi)^n}\prod_{j=1}^na_j L_{k_j}\left(\frac{a_j|z_j-z^\prime_j|^2}{2}\right)\\ \times \exp\left(-\frac 14\sum_{k=1}^na_k(|z_k|^2+|z_k^\prime|^2- 2z_k\bar z_k^\prime) \right).
\end{multline}
Here $L_k=L^{(0)}_k$, $k\in\ZZ_+$, is the Laguerre polynomial: 
\[
L_{k}(x)=\frac{e^x}{k!}\frac{d^k}{dx^k}(e^{-x}x^{k})=\sum_{j=0}^{k}\binom{k}{k-j}\frac{(-x)^j}{j!}, \quad x\geq 0.
\]

For the lowest eigenvalue $\Lambda_0(x_0)$, the kernel of the projection $P_{\Lambda_0,x_0}$ is the Bergman kernel $\mathcal P_{x_0}\in C^\infty(T_{x_0}X\times T_{x_0}X)$ given by (see \cite[(4.1.84)]{ma-ma:book})
\begin{equation}
\label{eq3.3}
\mathcal P_{x_0}(Z,Z^\prime)=\frac{1}{(2\pi)^n}\prod_{j=1}^na_j \exp\left(-\frac 14\sum_{k=1}^na_k(|z_k|^2+|z_k^\prime|^2- 2z_k\bar z_k^\prime) \right).
\end{equation}

The following theorem is the analog of \cite[Theorem 4.18']{dai-liu-ma}, \cite[Theorem 4.2.1]{ma-ma:book} and \cite[ Theorem 4.3]{ko-ma-ma} for the (generalized) Bergman kernel.

\begin{theorem}[\cite{Kor22},Theorem 3]\label{th3.2}
There exists $\varepsilon\in (0,r_X)$ such that for any $x_0\in X$ and $Z,Z^\prime\in T_{x_0}X$, $|Z|, |Z^\prime|<\varepsilon$, the sequence $P_{p,I,x_0}(Z,Z^\prime)$ admits an asymptotic expansion as $p\to\infty$
\begin{equation}\label{eq3.4}
\frac{1}{p^n}P_{p,I,x_0}(Z,Z^\prime)\cong
\sum_{r=0}^\infty F_{r,x_0}(\sqrt{p} Z, \sqrt{p}Z^\prime)\kappa_{x_0}^{-\frac 12}(Z)\kappa_{x_0}^{-\frac 12}(Z^\prime)p^{-\frac{r}{2}}, 
\end{equation}
where the leading coefficient $F_{0,x_0}\in C^\infty(T_{x_0}X\times T_{x_0}X)$ is given by
\[
F_{0,x_0}(Z,Z^\prime)=\mathcal P_{I, x_0}(Z,Z^\prime),
\]
and for any $r\geq 0$, the coefficient $F_{r,x_0}\in C^\infty(T_{x_0}X\times T_{x_0}X)$ has the form
\[
F_{r,x_0}(Z,Z^\prime)=J_{r,x_0}(Z,Z^\prime)\mathcal P_{x_0}(Z,Z^\prime),
\]
where $\mathcal  P_{x_0}$ is the Bergman kernel given by \eqref{eq3.3} and $J_{r,x_0}$ is a polynomial in $Z, Z^\prime$, depending smoothly on $x_0$, with the same parity as $r$ and $\operatorname{deg} J_{r,x_0}\leq \kappa(I)+3r$, where $\kappa(I)=\max \{|\mathbf k| : \Lambda_{\mathbf k,\mu}\in I\}$. 

For any $j\in \mathbb N$, the remainder 
\[
R_{j,p,x_0}(Z,Z^\prime):=\frac{1}{p^n}P_{p,I,x_0}(Z,Z^\prime)
-\sum_{r=0}^jF_{r,x_0}(\sqrt{p} Z, \sqrt{p}Z^\prime)\kappa_{x_0}^{-\frac 12}(Z)\kappa_{x_0}^{-\frac 12}(Z^\prime)p^{-\frac{r}{2}}
\]
in the asymptotic expansion \eqref{eq3.4} satisfies the following condition. For any $m,m^\prime\in \mathbb N$, there exist positive constants $C$, $c$, $c_0$ and $M$ such that for any $p\geq 1$, $x_0\in X$ and $Z,Z^\prime\in T_{x_0}X$, $|Z|, |Z^\prime|<\varepsilon$, 
\begin{multline}\label{eq3.5}
\sup_{|\alpha|+|\alpha^\prime|\leq m}\Bigg|\frac{\partial^{|\alpha|+|\alpha^\prime|}}{\partial Z^\alpha\partial Z^{\prime\alpha^\prime}}R_{j,p,x_0}(Z,Z^\prime)\Bigg|_{C^{m^\prime}(X)}\\ 
\leq Cp^{-\frac{j-m+1}{2}}(1+\sqrt{p}|Z|+\sqrt{p}|Z^\prime|)^M\exp(-c\sqrt{p}|Z-Z^\prime|)+ O(e^{-c_0\sqrt{p}}).
\end{multline}
\end{theorem}
Here $C^{m^\prime}(X)$ is the $C^{m^\prime}$-norm for the parameter $x_0\in X$. 

\section{Proofs of the main results}\label{s4}
In this section, we prove the main results of the paper. 

\begin{proof}[Proof of Theorem~\ref{th1.2}] 
The extension of the formula \eqref{eq2.3} in the current setting reads as follows:  
\begin{equation}\label{eq4.1}
\mathbb E[\mathcal N_p[f]] = \int_X P_{p,I}(x, x) f(x) dv_X(x).
\end{equation}
Since $P_{p,I}(x, x)$ is uniformly bounded and $f\in L^1(X,dv_X)$, the expectation $\mathbb E[\mathcal N_p[f]]$ is finite. 

Setting $Z=Z^\prime=0$ in \eqref{eq3.4} and using \eqref{eq3.1} and \eqref{eq3.2}, we get
\begin{multline*}
\frac{1}{p^n}P_{p,I}(x,x)=\frac{1}{p^n}P_{p,I,x}(0,0)=\mathcal P_{I,x}(0,0)+\mathcal O(p^{-1})\\ =\frac{1}{(2\pi)^n}|\mathcal K_I|\prod_{j=1}^na_j(x)+\mathcal O(p^{-1}), \quad p\to \infty,
\end{multline*}
uniformly on $x\in X$. 
Plugging this formula in \eqref{eq4.1} and using the fact that $\prod_{j=1}^na_j(x)dv_X(x)= \Omega_{\mathbf B}(x)$, we prove \eqref{eq1.5}.

Now assume that $f \in C^1_c(X, \mathbb R)$. The extension of the formula \eqref{eq2.4} in the current setting reads as follows: 
\begin{equation}\label{eq4.2}
\mathbb V[\mathcal N_p[f]] = \frac 12 \int_X \int_X |P_{p,I}(x, y)|^2_{h^{L^p}} (f(x)-f(y))^2\, dv_X(x)\, dv_X(y) . 
\end{equation}
Fix $x\in X$. By Theorem~\ref{th3.1}, it is easy to see that, for any $\epsilon>0$,  
\begin{multline*}
\int_{X\setminus B^{X}(x,p^{-\frac 12+\epsilon})} |P_{p,I}(x, y)|^2_{h^{L^p}} (f(x)-f(y))^2\, dv_X(y)
\\
\leq C_0 p^{2n} \int_{X\setminus B^{X}(x,p^{-\frac 12+\epsilon})}e^{-2c\sqrt{p} \,d(x,y)}\, dv_X(y)
=\mathcal O(p^{-\infty}).
\end{multline*}
All the estimates here and below are uniform on $x\in X$. Thus, we have 
\begin{multline*}
\frac 12 \int_X |P_{p,I}(x, y)|^2_{h^{L^p}} (f(x)-f(y))^2\, dv_X(y)\\
= \frac 12 \int_{B^{X}(x,p^{-\frac 12+\epsilon})} |P_{p,I}(x, y)|^2_{h^{L^p}} (f(x)-f(y))^2\, dv_X(y)+\mathcal O(p^{-\infty}).
\end{multline*}
Make a change of variables $y=\exp^X_x(Z)$ in the integral and denote $f_x(Z)=f(\exp^X_x(Z))$. We get 
\begin{multline*}
\frac 12 \int_{B^{X}(x,p^{-\frac 12+\epsilon})} |P_{p,I}(x, y)|^2_{h^{L^p}} (f(x)-f(y))^2\, dv_X(y)\\
=  \frac 12 \int_{B^{T_xX}(0,p^{-\frac 12+\epsilon})} |P_{p,I,x}(0, Z)|^2 (f_x(0)-f_x(Z))^2\kappa_x(Z)\, dv_{TX}(Z).
\end{multline*}
By Theorem~\ref{th3.2}, for any $x\in X$ and $Z\in T_{x}X$, $|Z|<\varepsilon$, 
\[
\frac{1}{p^n}P_{p,I,x}(0,Z)= \mathcal P_{I, x}(0, \sqrt{p}Z)\kappa_{x}^{-\frac 12}(Z)+R_{0,p,x}(0,Z).
\]
Using \eqref{eq3.5}, it is easy to check that 
\begin{multline*}
\frac 12 \int_{B^{T_xX}(0,p^{-\frac 12+\epsilon})} |P_{p,I,x}(0, Z)|^2 (f_x(0)-f_x(Z))^2\kappa_x(Z)\, dv_{TX}(Z) \\
=\frac 12p^{2n} \int_{B^{T_xX}(0,p^{-\frac 12+\epsilon})} |\mathcal P_{I, x}(0, \sqrt{p}Z)|^2 (f_x(0)-f_x(Z))^2\, dv_{TX}(Z)+\mathcal O(p^{n-\frac{3}{2}}). 
\end{multline*}
Now we make the change of variables $\sqrt{p}Z=Z_1$, omitting the subscript $1$ for convenience: 
\begin{multline*}
\frac 12p^{2n}\int_{B^{T_xX}(0,p^{-\frac 12+\epsilon})} |\mathcal P_{I, x}(0, \sqrt{p}Z)|^2 (f_x(0)-f_x(Z))^2\, dv_{TX}(Z)+\mathcal O(p^{n-\frac{3}{2}})\\
=\frac 12p^{n} \int_{B^{T_xX}(0,p^{\epsilon})} |\mathcal P_{I, x}(0, Z)|^2  (f_x(0)-f_x(Z/\sqrt{p}))^2\, dv_{TX}(Z)+\mathcal O(p^{n-\frac{3}{2}}).
\end{multline*}
We can clearly replace $B^{T_xX}(0,cp^{\epsilon})$ with $T_xX$ in the last integral. Combining the above identities, we get
\begin{multline*}
\frac 12 \int_X|P_{p,I}(x, y)|^2_{h^{L^p}} (f(x)-f(y))^2\, dv_X(y)\\ =\frac 12p^{n}\int_{T_xX} |\mathcal P_{I, x}(0, Z)|^2  (f_x(0)-f_x(Z/\sqrt{p}))^2\, dv_{TX}(Z)+\mathcal O(p^{n-\frac{3}{2}}).
\end{multline*}
Now we use the Taylor formula, taking into account that $d(\exp^X_x)(0): T_xX\to T_xX$ is the identity map:
$$
f_x(0)-f_x(Z)=-[df_x(0)](Z)+\mathcal O\left(|Z|^2\right)=-[df(x)](Z)+\mathcal O\left(|Z|^2\right), \quad Z\in T_xX.
$$
We get
\begin{multline*}
\frac 12 \int_X|P_{p,I}(x, y)|^2_{h^{L^p}} (f(x)-f(y))^2\, dv_X(y)\\ =\frac 12p^{n-1}\left(\int_{T_xX} |\mathcal P_{I, x}(0, Z)|^2  \left([df(x)](Z)\right)^2\, dv_{TX}(Z)+\mathcal O(p^{-\frac{1}{2}})\right).
\end{multline*}

Compute the integral in the right-hand side of the last formula:
\[
\mathcal J:=\int_{T_xX} |\mathcal P_{I, x}(0, Z)|^2  \left([df(x)](Z)\right)^2\, dv_{TX}(Z). 
\]
Let us use the linear isomorphism $T_{x}X\cong \mathbb R^{2n}$ determined by the orthonormal base $\{e_j : j=1,\ldots,2n\}$ in $T_{x}X$ (see Section~\ref{s3}). Using \eqref{eq3.1}, \eqref{eq3.2}, and the fact that the integral of an odd function vanishes, we proceed as follows:
\begin{align*}
\mathcal J=&\int_{\RR^{2n}} |\mathcal P_{I,x}(0, Z)|^2\left(\sum_{j=1}^{2n}Z_j\nabla_{e_{j}}f(x)\right)^2\, dZ\\
=& \left(\frac{1}{(2\pi)^n}\prod_{j=1}^na_j\right)^2 \sum_{\mathbf k^\prime,\mathbf k^{\prime\prime}\in \mathcal K_I}\int_{\RR^{2n}}\prod_{j=1}^nL_{k^\prime_{j}}\left(\frac{a_{j}|z_{j}|^2}{2}\right) L_{k^{\prime\prime}_{j}}\left(\frac{a_{j}|z_{j}|^2}{2}\right)\\ & \times \left(\sum_{\ell=1}^{2n}Z^2_\ell\left(\nabla_{e_{\ell}}f(x)\right)^2\right) \exp\left(-\frac 12\sum_{k=1}^na_k|z_k|^2 \right) dZ. 
\end{align*}
Next, we make the change of variables $Z_{2j-1}=r_j\cos\phi_j, Z_{2j}=r_j\sin\phi_j, j=1,\ldots,n$: 
\begin{align*}
\mathcal J =&\left(\frac{1}{(2\pi)^n}\prod_{j=1}^na_j\right)^2 (2\pi)^{n-1} \sum_{\mathbf k^\prime,\mathbf k^{\prime\prime}\in \mathcal K_I}\sum_{m=1}^{n}  \int_{\RR^{n}_+}\prod_{j=1}^nL_{k^\prime_{j}}\left(\frac{a_{j}r_{j}^2}{2}\right) L_{k^{\prime\prime}_{j}}\left(\frac{a_{j}r_{j}^2}{2}\right)\\
& \times \int_0^{2\pi}\left[\cos^2\phi_m\left(\nabla_{e_{2m-1}}f(x)\right)^2+\sin^2\phi_{m}\left(\nabla_{e_{2m}}f(x)\right)^2\right]d\phi_{m} \\
& \times r_m^2 \exp\left(-\frac 12\sum_{k=1}^na_kr_k^2 \right) \prod_{j=1}^nr_j dr_j\\
=& \left(\frac{1}{(2\pi)^n}\prod_{j=1}^na_j\right)^2 \frac{1}{2}(2\pi)^{n} \sum_{\mathbf k^\prime,\mathbf k^{\prime\prime}\in \mathcal K_I}\sum_{m=1}^{n}  \int_{\RR^{n}_+}\prod_{j=1}^nL_{k^\prime_{j}}\left(\frac{a_{j}r_{j}^2}{2}\right) L_{k^{\prime\prime}_{j}}\left(\frac{a_{j}r_{j}^2}{2}\right)\\
& \times \left[\left(\nabla_{e_{2m-1}}f(x)\right)^2+\left(\nabla_{e_{2m}}f(x)\right)^2\right]  
 r_m^2 \exp\left(-\frac 12\sum_{k=1}^na_kr_k^2 \right) \prod_{j=1}^nr_j dr_j.
 \end{align*}
Making another change of variables $y_k=\frac 12a_kr_k^2, k=1,\ldots,n,$ we arrive at
\[
\mathcal J
= \frac{1}{(2\pi)^n}\prod_{j=1}^na_j \sum_{\mathbf k^\prime,\mathbf k^{\prime\prime}\in \mathcal K_I}\sum_{m=1}^{n} \frac{1}{a_m}  I_m(\mathbf k^\prime,\mathbf k^{\prime\prime}) \left[\left(\nabla_{e_{2m-1}}f(x)\right)^2+\left(\nabla_{e_{2m}}f(x)\right)^2\right],   
\]
where
\[ 
 I_m(\mathbf k^\prime,\mathbf k^{\prime\prime})=  \int_{\RR^{n}} y_m  \prod_{j=1}^n L_{k^\prime_{j}}\left(y_{j} \right)L_{k^{\prime\prime}_{j}}\left(y_{j}\right)  e^{-y_j}dy_1\ldots dy_n.
\] 

To show that $I_m(\mathbf k^\prime,\mathbf k^{\prime\prime})$ is given by \eqref{eq1.4}, we recall the orthogonality conditions 
\[
\int_0^{+\infty}L_k(x)L_\ell(x)e^{-x}dx=\delta_{k,\ell}
\]
and recurrence relations
\[
xL_k(x)=(2k+1)L_k(x)-kL_{k-1}(x)-(k+1)L_{k+1}(x),
\]
which imply
\[
\int_0^{+\infty}xL_k(x)L_\ell(x)e^{-x}dx=(2k+1)\delta_{k,\ell}-k\delta_{k-1,\ell}-(k+1)\delta_{k+1,\ell}.
\]
We infer that 
\begin{align*}
I_m(\mathbf k^\prime,\mathbf k^{\prime\prime}) = & \left( \prod_{j\neq m} \int_0^{+\infty}  L_{k^\prime_{j}}\left(y_{j} \right)L_{k^{\prime\prime}_{j}}\left(y_{j}\right)e^{-y_j}dy_j\right)\\ & \times  \int_0^{+\infty}  y_m L_{k^\prime_m}\left(y_m \right)L_{k^{\prime\prime}_m}\left(y_m\right) e^{-y_m}dy_m\\
 = & \prod_{j\neq m} \delta_{k^\prime_{j},k^{\prime\prime}_{j}}((2k^\prime_m+1)\delta_{k^\prime_m,k^{\prime\prime}_m}-k^\prime_m\delta_{k^\prime_m-1,k^{\prime\prime}_m}-(k^\prime_m+1)\delta_{k^\prime_m+1,k^{\prime\prime}_m}),
\end{align*} 
that proves \eqref{eq1.4}.

We conclude that 
\[
\mathcal J=\frac{1}{(2\pi)^n}\prod_{j=1}^na_j  |df(x)|_{I}^2
\]
and 
\begin{multline*}
\frac 12 \int_X |P_{p,I}(x, y)|^2_{h^{L^p}} (f(x)-f(y))^2\, dv_X(y)\\
=\frac 12p^{n-1}\left( \frac{1}{(2\pi)^n}\prod_{j=1}^na_j  |df(x)|_{I}^2+\mathcal O(p^{-\frac{1}{2}})\right).  
\end{multline*}
In view of \eqref{eq4.2}, integrating the last identity against $dv_X(x)$ and using the fact that $\prod_{j=1}^na_j(x)dv_X(x)= \Omega_{\mathbf B}(x)$, we prove \eqref{eq1.6}.
\end{proof}

In the proof of Theorem~\ref{th1.2}, taking into account that $\prod_{j=1}^na_j(x)=\sqrt{\det B_x}$, we also derived the formula
\begin{equation}\label{eq4.3}
|df(x)|_{I}^2=\frac{(2\pi)^n}{\sqrt{\det B_x}}\int_{T_xX} |\mathcal P_{I,x}(0, Z)|^2\left(df(x)(Z)\right)^2\, dv_{TX}(Z),
\end{equation}
that shows that the function $x\in X\mapsto |df(x)|_{I}^2$ is a well-defined continuous function on $X$, which is smooth if $f \in C^\infty(X)$.

\begin{proof}[Proof of Theorem~\ref{th1.3}] 
By \cite[Theorem 0.6]{Demailly85} (see also \cite[Corollary 3.3]{Demailly91}), the asymptotic behavior of $N_p$ as $p\to\infty$ is given by
\[
N_p\sim \frac{p^n}{(2\pi)^n} \int_X N(x,I)\Omega_{\mathbf B}(x), \quad p\to \infty,  
\]
where
\[
N(x,I)=\#\{\mathbf k : \Lambda_{\mathbf k}(x)\in I\}, \quad x\in X.
\]
Under current assumptions on $I$, we have $N(x,I)\equiv |\mathcal K_I|$, and 
\begin{equation}\label{eq4.4}
N_p\sim \frac{p^n}{(2\pi)^n}|\mathcal K_I|{\rm Vol}_{\mathbf B}(X), \quad p\to \infty.  
\end{equation}

By Theorem~\ref{th1.2} and \eqref{eq4.4}, for any $f\in L^1(X,dv_X)$, we have
\[
\lim_{p\to \infty} \mathbb E\left[\frac{1}{N_p}\mathcal N_p[f]\right] = \frac{1}{{\rm Vol}_{\mathbf B}(X)} \int_X f(x) \Omega_{\mathbf B}(x). 
\] 
Using \eqref{eq2.4}, we get
\[
\mathbb V[\mathcal N_p[f]] = \frac 12 \int_X \int_X |P_{p,I}(x, y)|^2_{h^{L^p}} (f(x)-f(y))^2\, dv_X(x)\, dv_X(y)
\]
\[
\leq 2 \sup_{x\in X} |f(x)|^2\int_X \int_X P_{p,I}(x, y)P_{p,I}(y, x) dv_X(x)\, dv_X(y) 
\]
\[
= 2 \sup_{x\in X} |f(x)|^2\int_X P_{p,I}(x, x)\, dv_X(x)= 2 \sup_{x\in X} |f(x)|^2N_p. 
\]
It follows that
\[
\mathbb V\left[\frac{1}{N_p}\mathcal N_p[f]\right]=\frac{1}{N^2_p} \mathbb V\left[\mathcal N_p[f]\right]\leq Cp^{-n}.
\]
The convergence in probability \eqref{eq1.8} then follows from the classical Chebyshev inequality
\[
\mathbb P(|Y-\mathbb E(Y)|>\varepsilon)\leq \frac{\mathbb V(Y)}{\varepsilon^2},
\] 
as in the standard proof of the weak law of large numbers for random variables with finite variance.
\end{proof}

\begin{proof}[Proof of Theorem~\ref{th1.4}]
This theorem is a straightforward consequence of Theorem \ref{th1.2}, the formula \eqref{eq4.4} and the following general result due to Soshnikov.

\begin{thm}[\cite{Soshnikov02}, Theorem 1] 
Let $(X,\mathcal F,P_L)$, $L \geq 0$, be a family of determinantal random point fields with Hermitian correlation kernels $K_L$. Suppose that $f_L,L \geq 0$, are bounded measurable functions with compact support such that
\[
\mathbb V_L[S_{f_L}]\to \infty, \quad L\to\infty,
\]
and
\[
\sup |f_L(x)| = o(\mathbb V_L [S_{f_L}])^\varepsilon,\quad  \mathbb E_L S_{|f|_L}=O\left((\mathbb V_L [S_{f_L}])^\delta\right)
\]
for any $\varepsilon > 0$ and some $\delta > 0$. Then the normalized linear statistics $$\frac{S_{f_L}- \mathbb E_L [S_{f_L}]}{\sqrt{\mathbb V_L [S_{f_L}]}}$$ converges in distribution to the standard normal law $N(0, 1)$.
\end{thm}

The validity of these assumptions in the present setting when $n \geq 2$ (with $\delta =n/(n-1)$), follows directly from the asymptotics of the expectation and variance given by \eqref{eq1.5} and \eqref{eq1.6}.
\end{proof}

\section{Some remarks on the formula \eqref{eq1.6}}\label{s5}
In this section, we consider some particular cases of the asymptotic formula \eqref{eq1.6} for the variance. 

In the case when $|\mathcal K_I|=1$, that is, $\mathcal K_I$ consists of a single element $\mathbf k\in \mathbb Z_+^n$, we have 
\[
\alpha_m=I_m(\mathbf k,\mathbf k)=2k_m+1. 
\]
It follows that
\[
|df(x)|_{I}^2= \sum_{m=1}^{n} \frac{2k_m+1}{a_m(x)}   \left[\left(\nabla_{e_{2m-1}}f(x)\right)^2+\left(\nabla_{e_{2m}}f(x)\right)^2\right].
\]

In particular, when $\mathcal K_I$ consists of a single element $\mathbf k=0$, we have 
\[
|df(x)|_{I}^2= \sum_{m=1}^{n} \frac{1}{a_m(x)}   \left[\left(\nabla_{e_{2m-1}}f(x)\right)^2+\left(\nabla_{e_{2m}}f(x)\right)^2\right],
\]
that has the following geometric interpretation. 

Recall that $B_x : T_xX\to T_xX$, $x\in X$, be the skew-symmetric operator such that 
\[
\mathbf B_x(u,v)=g(B_xu,v), \quad u,v\in T_xX. 
\]
Let $J_x : T_xX\to T_xX$ be defined as $J_x=B_x(B^*_xB_x)^{-1/2}$. Then $J_x$ is an almost complex structure on $T_xX$ such that
\[
g(J_xu,J_xv)=g(u,v), \quad \mathbf B_x(J_xu,J_xv)=g(u,v), \quad u,v\in T_xX, 
\]
and the formula 
\[
g_{\mathbf B}(u,v)=\mathbf B_x(u,J_xv), \quad u,v\in T_xX, 
\]
defines a Riemannian metric on $TX$, called the Riemannian metric associated with $\mathbf B$ and $g$. It is clear that 
\[
g_{\mathbf B}(u,v)=g((B^*_xB_x)^{1/2}u,v), \quad u,v\in T_xX.
\]
Denote by $g^{-1}_{\mathbf B}$ the induced Riemannian metric on $T^*X$. Then it is easy to see that 
\[
|df(x)|_{I}^2=|df(x)|_{g^{-1}_{\mathbf B}}^2
\]
and 
\[
\mathbb V[\mathcal N_p[f]]=\frac{1}{4\pi}\frac{p^{n-1}}{(2\pi)^{n-1}}\int_X|df(x)|_{g^{-1}_{\mathbf B}}^2\,\Omega_{\mathbf B}(x)+ o(p^{n-1}), \quad p\to \infty.     
\]

Consider the case when $$a_1(x)=a_2(x)=\ldots=a_n(x)\left(=\frac{1}{n}\operatorname{Tr}(|B_x|)\right),\quad x\in X.$$
Then 
\[
\Sigma=\{2N+n: N\in \mathbb Z_+\}.
\]
Take a sufficiently small interval around $N$th Landau level $2N+n$, say, $$I_N=(2N+n-1,2N+n+1).$$ 
For the Landau Hamiltonian, this case corresponds to the pure $N+1$-analytic Ginibre point process (see \cite{HW19}). 

Then 
\[
\mathcal K_{I_N}:=\{\mathbf k\in \ZZ^n_+ : |\mathbf k|=N\}
\]
and 
\[
|\mathcal K_{I_N}|=\binom{N+n-1}{n-1}=\frac{(N+n-1)!}{N!(n-1)!}.
\]
It is clear that, for any $\mathbf k^\prime,\mathbf k^{\prime\prime}\in \mathcal K_I$, either $\mathbf k^\prime=\mathbf k^{\prime\prime}$ or $|\mathbf k^\prime-\mathbf k^{\prime\prime}|>1$. Therefore, we have
\[
\alpha_m=\sum_{|\mathbf k|=N} I_m(\mathbf k,\mathbf k)=\sum_{|\mathbf k|=N} (2k_m+1)
\]
It is easy to see that 
\[
\alpha_1=\alpha_2=\ldots=\alpha_n
\]
and 
\[
\alpha_1+\alpha_2+\ldots+\alpha_n=\sum_{|\mathbf k|=N} (2|\mathbf k|+n)=(2N+n)|\mathcal K_{I_N}|.
\]
We infer that
\begin{equation}\label{eq5.1}
\alpha_m=\frac{2N+n}{n}|\mathcal K_{I_N}|, \quad m=1,\ldots,n.
\end{equation}
and, therefore,
\begin{multline*}
|df(x)|_{I}^2
= (2N+n)|\mathcal K_{I_N}|\sum_{m=1}^{n} \frac{1}{\operatorname{Tr}(|B_x|)}\left[\left(\nabla_{e_{2m-1}}f(x)\right)^2+\left(\nabla_{e_{2m}}f(x)\right)^2\right]\\
= (2N+n)|\mathcal K_{I_N}|\frac{1}{\operatorname{Tr}(|B_x|)}|df(x)|^2_{g^{-1}}=\frac{2N+n}{n}|\mathcal K_{I_N}||df(x)|^2_{g^{-1}_{\mathbf B}}.
\end{multline*}
We get
\begin{equation}\label{eq5.2}
\mathbb V[\mathcal N_p[f]]=\frac{1}{4\pi}\frac{p^{n-1}}{(2\pi)^{n-1}}\frac{2N+n}{n}|\mathcal K_{I_N}| \int_X|df(x)|^2_{g^{-1}_{\mathbf B}}\,\Omega_{\mathbf B}(x)+ o(p^{n-1}).    
\end{equation}

Take $I=(n-1,2N+n+1)$. For the Landau Hamiltonian, this case corresponds to the full $N+1$-analytic Ginibre point process (see \cite{HW19}). We have
\[
\alpha_m=\sum_{j=0}^N\sum_{|\mathbf k|=j} I_m(\mathbf k,\mathbf k)+2\sum_{j=1}^N \sum_{|\mathbf k|=j} I_m(\mathbf k,\mathbf k-e_m).
\]
By \eqref{eq5.1}, we have
\[
\sum_{|\mathbf k|=j} I_m(\mathbf k,\mathbf k-e_m)=-\sum_{|\mathbf k|=j}k_m=-\frac{j}{n}|\mathcal K_{I_j}|.
\]
It follows that
\[
\alpha_m=\sum_{j=0}^N\frac{2j+n}{n}|\mathcal K_{I_j}|-2\sum_{j=1}^N \frac{j}{n}|\mathcal K_{I_j}|=\sum_{j=0}^N|\mathcal K_{I_j}|.
\]
We have
\[
\sum_{j=0}^N|\mathcal K_{I_j}|=\#\{\mathbf k\in \ZZ^n_+ : |\mathbf k|\leq N\}. 
\]
It is easy to see that 
\[
\#\{\mathbf k\in \ZZ^n_+ : |\mathbf k|\leq N\}=\#\{\mathbf k^\prime\in \ZZ^{n+1}_+ : |\mathbf k^\prime|= N\}=\binom{N+n}{n}=\frac{N+n}{n}|\mathcal K_{I_N}|.
\]
It follows that
\[
|df(x)|_{I}^2=\frac{N+n}{n}|\mathcal K_{I_N}||df(x)|^2_{g^{-1}_{\mathbf B}}
\]
and
\begin{multline}\label{eq5.3}
\mathbb V[\mathcal N_p[f]]=\frac{1}{4\pi}\frac{p^{n-1}}{(2\pi)^{n-1}}\frac{N+n}{n}|\mathcal K_{I_N}|\\ \times  \int_X|df(x)|^2_{g^{-1}_{\mathbf B}}\,\Omega_{\mathbf B}(x)+ o(p^{n-1}), \quad p\to \infty.     
\end{multline}

So we see that the variance in \eqref{eq5.2} is higher than in \eqref{eq5.3}. In fact, 
the variance in \eqref{eq5.3} is obtained by averaging the variances from each of the $N+1$ Landau levels, since
\[
\frac{1}{N+1}\sum_{j=0}^{N}(2j+n)=N+n.
\] 
This fact was observed in the polyanalytic case in \cite{HW19}.


\begin{thebibliography}{00}
\bibitem{Abreu23}
L. D. Abreu, Entanglement entropy and hyperuniformity of Ginibre and Weyl-Heisenberg ensembles. \textit{Lett. Math. Phys.} \textbf{113} (2023), no. 3, Paper No. 54, 14 pp.

\bibitem{APRT17}
L. D. Abreu, J. M. Pereira, J. L. Romero, and S. Torquato,  The Weyl-Heisenberg ensemble: hyperuniformity and higher Landau levels. \textit{J. Stat. Mech. Theory Exp.} \textbf{2017}, no. 4, Paper No. 043103, 16 pp.

\bibitem{AGZ10} G. W. Anderson, A. Guionnet, and O. Zeitouni, \textit{An introduction to random matrices,}
Cambridge Studies in Advanced Mathematics, vol. 118, Cambridge University Press, Cambridge, 2010. 

\bibitem{Berman14} 
R.J. Berman, Determinantal point processes and fermions on complex manifolds: large deviations
and bosonization. \textit{Comm. Math. Phys.} \textbf{327} (2014), 1--47. 

\bibitem{Berman18} 
R.J. Berman, Determinantal point processes and fermions on polarized complex manifolds: bulk universality, In \textit{Algebraic and analytic microlocal analysis,} Springer Proc. Math. Stat., vol. 269, Springer, Cham, 2018, pp. 341--393.

\bibitem{charles20b}
L. Charles, On the spectrum of nondegenerate magnetic Laplacians. \textit{Anal. PDE} \textbf{17} (2024), no. 6, 1907--1952.

\bibitem{Charles-Estienne20}
L. Charles and B. Estienne, Entanglement entropy and Berezin-Toeplitz operators. \textit{Comm.
Math. Phys.} \textbf{376} (2020), 521--554.

\bibitem{dai-liu-ma}
X. Dai, K. Liu, and X. Ma, 
On the asymptotic expansion of Bergman kernel. \textit{J. Differential Geom.} \textbf{72} (2006), 1--41.

\bibitem{DL25a}
A. Deleporte and G. Lambert,
Fluctuations of two-dimensional determinantal processes associated with Berezin-Toeplitz operators, preprint  arXiv:2506.11707 (2025).

\bibitem{Demailly85}
J.-P. Demailly, Champs magn\'{e}tiques et in\'{e}galit\'{e}s de {M}orse pour la {$d''$}-cohomologie \textit{Ann. Inst. Fourier (Grenoble)} \textbf{35} (1985), no. 4, 189--229.

\bibitem{Demailly91}
J.-P. Demailly, Holomorphic Morse inequalities. In \textit{Several complex variables and complex geometry, Part 2 (Santa Cruz, CA, 1989)}, Proc. Sympos. Pure Math., 52, Part 2, Amer. Math. Soc., Providence, RI, 1991, pp. 93--114.

\bibitem{FL22}
M. Fenzl and G. Lambert. Precise deviations for disk counting statistics of invariant
determinantal processes. \textit{Int. Math. Res. Not. IMRN,} \textbf{2022}, no. 10, 7420--7494. 

\bibitem{Ginibre65}
J. Ginibre, Statistical ensembles of complex, quaternion, and real matrices, \textit{J. Math. Phys.,} \textbf{6}
(1965), 440--449.

 \bibitem{Gu-Uribe}
V. Guillemin and A. Uribe, The Laplace operator on the $n$th tensor power of a line bundle: eigenvalues which are uniformly bounded in $n$. \textit{Asymptotic Anal.} \textbf{1} (1988), 105--113.


\bibitem{HH13}
A. Haimi and H. Hedenmalm, The polyanalytic Ginibre ensembles, \textit{J. Stat. Phys.,} \textbf{153} (2013), 10--47.

\bibitem{HW19}
A. Haimi and A. Wennman, A central limit theorem for fluctuations in Polyanalytic Ginibre ensembles. \textit{Int. Math. Res. Not. IMRN,} \textbf{2019}, no. 5,  1350--1372.

\bibitem{HKPV09} J. B. Hough, M. Krishnapur, Yu. Peres, and B. Vir\'{a}g, {Zeros of Gaussian analytic functions and determinantal point processes,} University Lecture Series, vol. 51,
American Mathematical Society, Providence, RI, 2009.

\bibitem{Ioos} 
L. Ioos, Partial Bergman kernels and determinantal point processes on K\"ahler manifolds, preprint arXiv:2511.20539 (2025), 35 pp.

\bibitem{Kor22}  
Yu. A.  Kordyukov, Semiclassical spectral analysis of the Bochner-Schr\"odinger operator on symplectic manifolds of bounded geometry, \textit{Anal. Math. Phys.} \textbf{12} (2022). no. 1, Paper No 22, 37 pp.

\bibitem{ko-ma-ma} Yu. A. Kordyukov, X. Ma, and G. Marinescu, 
Generalized Bergman kernels on symplectic manifolds of bounded geometry. \textit{Comm. Partial Differential Equations} \textbf{44} (2019), 1037--1071.

\bibitem{Lemoine22}
T. Lemoine. Determinantal point processes associated with Bergman kernels: construction
and limit theorems, preprint arXiv:2211.06955 (2022), 

\bibitem{ma-ma:book}
X. Ma and G. Marinescu, \textit{Holomorphic Morse inequalities and Bergman kernels.} Progress in Mathematics, 254. Birkh\"auser Verlag, Basel, 2007. 

\bibitem{ma-ma08} 
X. Ma and G. Marinescu, Generalized Bergman kernels on symplectic manifolds.\textit{Adv. Math.} \textbf{217} (2008), 1756--1815.

\bibitem{ma-ma15} X.~Ma and G.~Marinescu, 
Exponential estimate for the asymptotics of Bergman kernels. \textit{Math. Ann.} \textbf{362} (2015), no. 3-4, 1327--1347.

\bibitem{Mac75}
O. Macchi, The coincidence approach to stochastic point processes, \textit{Advances in Appl. Probability} \textbf{7} (1975), 83--122.  

\bibitem{Mac77} O. Macchi, The Fermion process --- a model of stochastic point process with repulsive points, In \textit{Transactions of the Seventh Prague Conference on Information Theory, Statistical Decision
Functions, Random Processes and of the Eighth European Meeting of Statisticians (Tech.
Univ. Prague, Prague, 1974), Vol. A,} Reidel, Dordrecht-Boston, Mass., 1977, pp. 391--398.

\bibitem{MMM23}
Z. Mouayn, M. Mahboubi, and O. El Moize, Mean and variance of the cardinality of particles in infinite true polyanalytic Ginibre processes via a coherent states quantization method. \textit{J. Stat. Mech. Theory Exp.} \textbf{2023}, no. 7, Paper No. 073103, 21 pp.

\bibitem{RV07}
B. Rider and B. Vir\'{a}g, The noise in the circular law and the Gaussian free field. \textit{Int. Math. Res. Not. IMRN} \textbf{2007}, no. 2, Art. ID rnm006, 33 pp.

\bibitem{shirai15}
T. Shirai, Ginibre-type point processes and their asymptotic behavior. \textit{J. Math. Soc. Japan,} \textbf{67} (2015), 763--787.

\bibitem{SoshnikovUMN}  
A. Soshnikov, Determinantal random point fields. (Russian); translated from \textit{Uspekhi Mat. Nauk} \textbf{55} (2000), no. 5(335), 107--160 \textit{Russian Math. Surveys} \textbf{55} (2000), no. 5, 923--975

\bibitem{Soshnikov02}  
A. Soshnikov, Gaussian limit for determinantal random point fields, \textit{Ann. Probab.} {30} (2002), no. 1, 171--187.

\end{thebibliography}
\end{document}